\title{\vspace{-0.6cm} On tight cycles in hypergraphs}

\documentclass[11pt]{article}

\usepackage{amsmath, amssymb, amsthm, url, tikz, verbatim}

\newtheorem{theorem}{Theorem}[section]
\newtheorem{proposition}[theorem]{Proposition}
\newtheorem{lemma}[theorem]{Lemma}

\newtheorem{question}[theorem]{Question}

\newtheorem{conjecture}[theorem]{Conjecture}

\usetikzlibrary{calc}
\usetikzlibrary{patterns}
\usetikzlibrary{decorations.markings}
\usetikzlibrary{arrows,shapes.geometric,%
	decorations.pathreplacing,shapes,shadows}

\date{}

\author{
Hao Huang \thanks{
Department of Math and CS, Emory University, Atlanta, GA 30322, USA.
Email: hao.huang@emory.edu. Research supported in part by the Collaboration Grants from the Simons Foundation.
}
\and
Jie Ma\thanks{
School of Mathematical Sciences,
University of Science and Technology of China, Hefei, Anhui 230026, China.
Email: jiema@ustc.edu.cn. Research supported in part by National Natural Science
Foundation of China (NSFC) grants 11501539 and 11622110.
}
}

\begin{document}
\maketitle
\abstract
A tight $k$-uniform $\ell$-cycle, denoted by $TC_\ell^k$, is a $k$-uniform hypergraph whose vertex set is $v_0, \cdots, v_{\ell-1}$, and the edges are all the $k$-tuples $\{v_i, v_{i+1}, \cdots, v_{i+k-1}\}$, with subscripts modulo $\ell$. Motivated by a classic result in graph theory that every $n$-vertex cycle-free graph has at most $n-1$ edges, S\'os and, independently, Verstra\"ete asked whether for every integer $k$, a $k$-uniform $n$-vertex hypergraph without any tight $k$-uniform cycles has at most $\binom{n-1}{k-1}$ edges. In this paper, we answer this question in negative. 


\section{Introduction} \label{section_introduction}

A classic result in graph theory says that if an $n$-vertex graph $G$ contains no cycle, then $G$ has at most $n-1$ edges. This upper bound is tight by the $n$-vertex trees. It is natural to ask whether such extremal result can be generalized to $k$-uniform hypergraphs. There are many different notions of cycles in $k$-uniform hypergraphs, most notably Berge cycles, loose cycles, and tight cycles (see \cite{verstraete_survey} for the definitions). They all coincide with the usual definitions of cycles when $k=2$. For the Tur\'an numbers on Berge cycles and loose cycles, there has been extensive research in the literature (see \cite{BG,CGJ,FJ,FO,G06,GL-3uniform,GL,JM,KMV}).

In this paper, we consider Tur\'an-type problems of the {\it tight $k$-uniform $\ell$-cycle}, denoted by $TC_\ell^k$  (for $\ell\ge k+1$), which is the $k$-uniform hypergraph whose vertex set is $v_0, \cdots, v_{\ell-1}$, and the edges are all the consecutive $k$-tuples $\{v_i, v_{i+1}, \cdots, v_{i+k-1}\}$, with subscripts modulo $\ell$.
It seems that the Tur\'an-type problems for tight cycles are more difficult, perhaps because of its interlocking structure. When the desired tight cycle has length linear in $n$, Allen et al. \cite{abcm} showed that for any $\delta >0$, there exists $n_0$ such that for any $0\le\alpha \le 1$, every $k$-uniform hypergraph on $n \ge n_0$ vertices with at least $(\alpha+\delta) {n \choose k}$ edges contains a tight cycle of length at least $\alpha n$. For $\alpha=1$ (i.e. Hamiltonian cycles), R\"odl, Ruci\'nski and Szemer\'edi \cite{rrs1, rrs2} established an approximate extension of Dirac's Theorem for $k$-uniform hypergraphs: for all $\gamma >0$, every sufficiently large $k$-uniform hypergraph such that every $(k-1)$-set of vertices lies in at least $(1/2+\gamma)n$ edges has a tight Hamilton
cycle. Both results uses the regularity lemma which is powerful when the extremal hypergraph is dense. Nevertheless, for shorter tight cycles, in particular cycles of constant lengths divisible by $k$, the extremal hypergraph is sparse and much less has been known.

Back to our initial question, we call a $k$-uniform hypergraph \textit{tight-cycle-free} if it does not contain any $TC_{\ell}^k$ for $\ell\ge k+1$. Throughout this paper we denote by $f_k(n)$ the maximum number of edges in a $k$-uniform $n$-vertex tight-cycle-free hypergraph. In this notation, $f_2(n)=n-1$. Indeed it is not hard to see that for every integer $k \ge 2$, $f_k(n) \ge \binom{n-1}{k-1}$. A $k$-uniform hypergraph is called a {\it full-$k$-star} if it consists of all the $k$-subsets of $[n]$ containing a given element called the {\it center}. When the uniformity is clear from the context, we will simply call it a full-star. One can easily verify that a full-$k$-star has $\binom{n-1}{k-1}$ edges and contains no tight cycles, since any tight cycle must have length at least $k+1$ and thus at least one of its edges cannot contain the center. S\'os and Verstraete (see \cite{verstraete_survey}) independently asked whether $f_k(n)$ always equals $\binom{n-1}{k-1}$. This conjecture is still wide open. In particular it is not even clear if $f_k(n)=\Theta(n^{k-1})$. One easy upper bound follows from an observation that $TC_{2k}^k$ is actually a sub-hypergraph of the complete $k$-partite $k$-uniform hypergraph $K^{k}_{2, \cdots, 2}$, whose Tur\'an number was shown to be $O(n^{k-(1/2)^{k-1}})$ by Erd\H os \cite{erdos_partite}. This gives $f_k(n) =O(n^{k-(1/2)^{k-1}})$. For $k=3$, currently the best known upper bound is $f_3(n)=O(n^{5/2})$ proved by Verstra\"ete \cite{verstraete_pc}, by showing that the Tur\'an number of $TC_{24}^3$ is $O(n^{5/2})$.

In this paper, we prove the following theorem, giving a negative answer to the S\'os-Verstra\"ete Conjecture.

\begin{theorem} \label{thm_main}
For every integer $k \ge 3$, there exits a constant $c=c(k)>0$ such that for sufficiently large $n$,
 $$f_k(n) \ge (1+c) \binom{n-1}{k-1}.$$
\end{theorem}



The rest of the paper is organized as follows. In the next section, we first present an explicit construction showing that $f_3(n)$ is strictly greater than $\binom{n-1}{2}$, but only better by a linear function in $n$. In Section \ref{section_proof}, we generalize this idea to construct denser hypergraphs without tight cycles and prove our main result Theorem \ref{thm_main}. 
The final section contains some concluding remarks and open problems.
Throughout this paper, we write $[n]$ for the set $\{1,2,\cdots, n\}$.

\section{Warm-up} \label{section_warmup}
As a warm-up, we construct a $3$-uniform $n$-vertex hypergraph $H$ with $\binom{n-1}{2} + \frac{n-1}{2}$ edges, so that $H$ does not contain any tight cycle.
The key idea is that one can slightly modify the full-$3$-star by removing some edges containing its center, and adding twice as many edges, while preserving the tight-cycle-freeness.

\begin{proposition} \label{prop_warmup}
For odd integers $n \ge 7$,
$$f_3(n) \ge \binom{n-1}{2} + \frac{n-1}{2}.$$
\end{proposition}
\begin{proof}
Let $H$ be the $3$-uniform hypergraph on the vertex set $[n]$ consisting of the following two types of edges. Type-I edges include all the triples containing $1$, except for those in the form $\{1, 2i, 2i+1\}$ for $i=1, \cdots, (n-1)/2$. Type-II edges are those in the form $\{2i, 2i+1, 2i+2\}$ or $\{2i, 2i+1, 2i+3\}$, for $i=1, \cdots \frac{n-3}{2}$, or the two triples $\{n-1, n, 2\}$ and $\{n-1, n, 3\}$. There are $\binom{n-1}{2}-(n-1)/2$ edges of Type I, and $n-1$ edges of Type II.
So $e(H)=\binom{n-1}{2}+(n-1)/2$. Notice that the only two edges in $H$ containing some $x\in \{2i,2i+1\}$ and $y\in \{2i+2,2i+3\}$ are $\{1,x,y\}$ and $\{2i,2i+1,y\}$.

Suppose $H$ contains a tight cycle. Such cycle must have length at least $4$. In particular, it contains at least one Type-II edge. Without loss of generality, we assume it is the triple $\{2,3,4\}$. Modulo symmetry, there are two possibilities of the relative order of these three vertices on the cycle: (i) they are in the order $2, 3, 4$; (ii) they are in the order $2, 4, 3$.

In the first case, since the only edges containing both $2$ and $3$ are the Type-II edges $\{2,3,4\}$ and $\{2,3,5\}$, and the only edges containing both $3$ and $4$ are $\{2,3,4\}$ and a Type-I edge $\{1, 3, 4\}$, it is not hard to see that the cycle must contain consecutive vertices $5,2,3,4,1$. Now the only edge (except $\{2,3,5\}$) containing $2$ and $5$ is $\{1,2,5\}$. This shows that the cycle consists of vertices $5,2,3,4,1$. However $\{1,4,5\} \not\in E(H)$, therefore Case (i) is impossible.

Similarly, for the second case, the cycle must consist of (consecutive) vertices $2, 4, 3, 1$ but $\{1,2,3\} \not \in E(H)$. Therefore both cases are not possible, and $H$ does not contain a tight cycle. This shows that $f_3(n) \ge \binom{n-1}{2} + \frac{n-1}{2}$.
\end{proof}

We remark that the above construction remains tight-cycle-free, if the range of the index $i$ is changed from $1\leq i\leq (n-1)/2$ to any integer set of size at least three.

\section{Proof of Theorem \ref{thm_main}} \label{section_proof}
In this section, we prove our main theorem, showing that for every integer $k \ge 3$, there exist tight-cycle-free $k$-uniform hypergraphs having $c$-fractional more edges than the full-$k$-star for some constant $c>0$.

\subsection{The key lemma}
The following lemma will be critical for the proof of Theorem \ref{thm_main}.
Loosely speaking, it can be used to reduce Theorem \ref{thm_main} to the problem of just finding any tight-cycle-free $k$-uniform hypergraphs with more edges than the full-$k$-star.

The {\it $t$-shadow} of a hypergraph $H$, denoted by $\partial_t(H)$, is defined as the following:
$$\partial_t(H)=\{S: |S|=t, S \subset e~\textrm{for some }e \in E(H)\}.$$
If $H$ consists of a single edge $e$, then we will write it as $\partial_t(e)$.

\begin{lemma}\label{lemma_shadow}
Let $H$ be a full-$k$-star with the vertex set $\{0\}\cup [n]$ and the center $0$.
Let $G_1, \cdots, G_t$ be subhypergraphs of $H$, and $F_1, \cdots, F_t$ be $k$-uniform hypergraphs on $[n]$.
Suppose they satisfy the following properties:
\begin{itemize}
\item [(i)] The hypergraph $H_i:=(H\setminus G_i) \cup F_i$ is tight-cycle-free for all $i =1, \cdots, t$, and
\item [(ii)] $\partial_{k-1}(F_i) \cap \partial_{k-1}(F_j)  = \emptyset$ for all $1\leq i< j\leq t$.
\end{itemize}
Then
$$H':=\left(H\setminus \bigcup_{i=1}^t G_i\right) \cup \left(\bigcup_{i=1}^t F_i\right)$$
is also tight-cycle-free.
\end{lemma}
\begin{proof}
We prove by contradiction. Suppose that $H'$ contains a tight cycle $C=v_1v_2 \cdots v_{\ell}$. If $0\not \in V(C)$, then all the edges of $C$ must come from $F_1\cup \cdots\cup F_t$. Note that the $(k-1)$-shadows of distinct $F_i$'s are disjoint, therefore all the edges of $C$ must come from the same $F_i$ (as, otherwise, there exist two consecutive edges of $C$ from distinct $F_i$'s which would share $k-1$ common vertices). However, from (i), each $H_i$ and in particular each $F_i$ is tight-cycle-free, which is impossible.

Now suppose $0\in V(C)$. Using similar arguments, we can show that the $l-k$ edges of $C$ not containing $0$ are from the same $F_i$. From the definition of $H'$, we see that the $k$ edges of $C$ containing $0$ must be in $H_i$. Therefore $C$ is completely contained in $H_i$, again contradicting that $H_i$ has no tight cycle.
\end{proof}

\subsection{The 3-uniform case}
We start with the case $k=3$. We will use the previous lemma to prove the following theorem.

\begin{theorem}\label{thm_3} For sufficiently large $n$,
$$f_3(n) \ge \left(1+\frac{1}{5}\right) \binom{n-1}{2}+O(n).$$
\end{theorem}

\begin{proof}
It suffices to show that if $n$ is sufficiently large and $n-7$ is divisible by $30$,
then $f_3(n)\geq (1+1/5)\binom{n-1}{2}$.
Consider the full-$3$-star with the vertex set $\{0,1,\cdots,n-1\}$ and the center $0$.
Let $F=\{123,124,345,346,561,562\}$ and $G=\{012,034,056\}$.
We call $G$ as the {\it accompanying set} of $F$.

By the remark in the end of Section 2, the $3$-uniform hypergraph obtained from the full-star by adding edges in $F$ and deleting edges in $G$ is tight-cycle-free.
This together with Lemma \ref{lemma_shadow} show that if one can find many copies of $F$ (say $F_1,\cdots, F_t$) on the vertex set $\{1,\cdots,n-1\}$ whose 2-shadows are pairwise disjoint,
then the hypergraph obtained from the full-$3$-star by adding $F_1\cup \cdots \cup F_t$ and deleting their accompanying sets is also tight-cycle-free.
And $e(H)=\binom{n-1}{2}+3t$.

Let $L$ be the 2-shadow of $F$, which is a 5-regular graph on 6 vertices, i.e. $K_6$. To find such $F$-copies, it is enough to find many edge-disjoint $L$-copies in the complete graph $K_{n-1}$ on $\{1,\cdots,n-1\}$.
By a famous theorem of Wilson \cite{W1,W2,W3,W4},
if $n$ is sufficiently large and satisfies $6|(n-1)$ and $5|(n-2)$,
then one can partition the edge set of $K_{n-1}$ into edge-disjoint copies of $L$.
Since $n-7$ is divisible by $30$, indeed we can apply this result to get
$$t=\frac{1}{e(L)}\binom{n-1}{2}=\frac{1}{15}\binom{n-1}{2}$$
copies of sub-hypergraphs isomorphic to $F$, whose 2-shadows are pairwise disjoint.
Therefore, by Lemma \ref{lemma_shadow} we can find a tight-cycle-free 3-graph $H$ with
$$e(H)=\binom{n-1}{2}+3t=(1+\frac15)\binom{n-1}{2},$$
completing the proof of Theorem \ref{thm_3}.
\end{proof}

It is not clear how to generalize the cyclic-type construction in Section \ref{section_warmup} to $k$-uniform hypergraphs for $k\geq 4$. For that, we will show a different construction beating the conjectured $\binom{n-1}{k-1}$ bound for the general case in the next subsection; then applying Lemma \ref{lemma_shadow} again would imply a construction with $(1+c) \binom{n-1}{k-1}$ edges for some $c>0$.

\subsection{The general case: $k$-uniform}


\begin{lemma}\label{lem_kstar1}
Let $H$ be a full-$k$-star with center $0$.
If we
\begin{itemize}
\item add new edges $e_1,\cdots, e_t\subseteq V(H)\setminus \{0\}$ such that $\partial_{k-1}(e_i)\cap \partial_{k-1}(e_j)=\emptyset$ for all $i\neq j$, and
\item delete edges $\{0\}\cup f_i$ for all $1\leq i\leq t$, where $f_i$ is any $(k-1)$-subset of $e_i$,
\end{itemize}
then the resulting $k$-uniform hypergraph $H'$ remains tight-cycle-free and has the same number of edges with the full-$k$-star.
\end{lemma}

\begin{proof}
Suppose for contradiction that $H'$ contains a tight cycle $C$. Then $C$ has to contain a new edge, say $e_1=\{1,\cdots,k\}$.
Since the $(k-1)$-shadows of $e_i$'s are disjoint, it is not hard to see that the only possibility to have a tight cycle $C$ is that
$V(C)=\{0,1,\cdots,k\}$. This also means that $C$ is a complete $k$-uniform hypergraph on $k+1$ vertices.
But this is impossible as we have deleted an edge $\{0\}\cup f_1$ for some $(k-1)$-subset $f_1\subseteq e_1$.
This proves the lemma.
\end{proof}

\begin{lemma}\label{lem_kstar+1}
For every integer $k\geq 3$, there exists a positive integer $m=m(k)$ such that the following holds for any $n\geq m$.
Let $H$ be a full-$k$-star with the vertex set $\{0\}\cup [n]$ and the center $0$.
Then there exist two $k$-uniform hypergraphs $F$ and $G$ such that
\begin{itemize}
\item $V(F)$ is a subset of size at most $m$ in $[n]$,
\item $G$ is a subhypergraph of $H$,
\item $H':=(H\setminus G)\cup F$ is tight-cycle-free with $e(H)+1$ edges.
\end{itemize}
\end{lemma}

\begin{proof}
Let $e_0=[k]$. For all $1\leq i\leq k$, let $f_i:=[k]\setminus \{i\}$ be a $(k-1)$-subset and $e_i:=f_i\cup \{k+i\}$.
We will construct $F=\{e_0,e_1,\cdots, e_{k+M}\}$ for some positive $M=O(k^3)$ and
$$G=\{\{0\}\cup f_i: f_i\subseteq e_i \text{ and } |f_i|=k-1 \text{ for every } 1\leq i\leq k+M\}$$ such that
$\partial_{k-1}(e_i)\cap \partial_{k-1}(e_j)=\emptyset$ for all $1\leq i<j\leq k+M$. Let $H'=(H \setminus G) \cup F$. Then $e(H')=e(H)
-e(G)+e(F)=e(H)+1$.

Suppose we have such $F$ and $G$. By Lemma \ref{lem_kstar1}, $(H\setminus G)\cup (F\setminus \{e_0\})$ is tight-cycle-free.
So all potential tight cycles $C$ in $H'$ must contain the edge $e_0$. We will discuss all possibilities for such $C$ as follows.
First since the full $k$-star is tight-cycle-free, $C$ must contain at least one edge from $\{e_0, e_1, \cdots, e_{k+M}\}$.
If only $e_0$ appears in $C$, then $V(C) =\{0,1,\cdots,k\}$, however this is impossible since all the edges $\{0\}\cup f_i$ ($1\leq i\leq k$) are deleted in $H'$.
Recall that we would choose $e_i$'s so that their $(k-1)$-shadows are pairwise disjoint.
Therefore $C$ has at least one edge in $\{e_1,\cdots, e_k\}$; and it is also not hard to see that the tight cycle $C$ can use at most two edges in $\{e_1,\cdots, e_k\}$.

Consider the case that $C$ contains exactly one edge in $\{e_1,\cdots, e_k\}$, say $e_\ell$ for some $\ell\in [k]$.
Then the cycle $C$ must have consecutive vertices $ \pi(1),\pi(2),\cdots, \pi(k), k+\ell$, where $\pi:[k]\to[k]$ is a permutation, $\pi(1)=\ell$ and $e_\ell=\{\pi(2),\cdots, \pi(k),k+\ell\}$. These vertices alone do not form a tight cycle, otherwise it would violate that $\partial_{k-1}(e_i)\cap \partial_{k-1}(e_j)=\emptyset$ for all $1\leq i<j\leq k+M$. Consider the vertex immediate before $\pi(1)$. For the same reason, this vertex can only be $0$, however this is also not possible since for all $(k-1)$-subset $f_j \subset [k]$, $\{0\} \cup f_j \not\in E(H')$.

Now we may assume that $C$ contains two edges in $\{e_1,\cdots, e_k\}$, say $e_i$ and $e_j$ for some $1\leq i<j\leq k$.
By similar analysis as above, the only possible tight cycle $C$  that may appear in $H'$ is on the vertex set in the order $k+i,\pi(1),\cdots, \pi(k),k+j,0$ for some permutation $\pi$ of $[k]$ with $\pi(1)=j$ and $\pi(k)=i$.
For all $1\leq i<j\leq k$ and all sets $\{\pi(4), \cdots, \pi(k)\}$ (there are at most ${k \choose 3}$ such $(k-3)$-sets), we let
$$f_{k+\alpha}=\{\pi(4), \cdots, \pi(k), k+j, k+i\}$$ be all such $(k-1)$-subsets (say $M$ of them in total), where $M\leq 3\binom{k}3=O(k^3)$, and we may label them by $1\leq \alpha\leq M$.

By removing the edges $\{0\} \cup f_{k+\alpha}$, we destroy all the tight cycle mentioned above. One can find $k$-subsets $e_{k+1},\cdots, e_{k+M}$ with pairwise distinct $(k-1)$-shadows such that $f_{k+\alpha} \subset e_{k+\alpha}$ for all $\alpha=1, \cdots, M$. For example, let $e_{k+\alpha} = f_{k+\alpha} \cup \{2k+\alpha\}$.
It remains to check that these new $e_{k+\alpha}$'s have distinct $(k-1)$-shadows from $e_1, \cdots, e_k$. This follows from the observation that $2k+j \not \in e_l$ for $1 \le l \le k$, and at least one of $\{k+j, k+i\}$ is also not in $e_l$.
We conclude this proof by noting that the value of $m$ can be $O(k^3)$.
\end{proof}


We are ready to prove our main result -- Theorem \ref{thm_main}. It is a combination of the construction in Lemma \ref{lem_kstar+1} and some classic hypergraph packing results.

\begin{proof}[Proof of Theorem \ref{thm_main}.]
Let $n$ be sufficiently large and $H$ be the full-$k$-star with center $0$ on $\{0, \cdots, n-1\}$.
By Lemma \ref{lem_kstar+1}, there exists a $k$-uniform tight-cycle-free hypergraph $H'=(H\setminus G)\cup F$ with $\binom{n-1}{k-1}+1$ edges,
where $F$ is a $k$-uniform hypergraph on at most $m=O(k^3)$ vertices with $V(F)\subseteq [n-1]$ and $G$ is a subhypergraph of $H$.
Let us call $G$ the {\it accompanying} hypergraph of $F$.
By Lemma \ref{lemma_shadow}, if one can find copies of sub-hypergraphs on $[n-1]$ isomorphic to $F$, denoted by $F_1, \cdots, F_t$, such that $\partial_{k-1}(F_i)$'s are pairwise disjoint, then
$$H''=\left(H \setminus \bigcup_{i=1}^t G_i \right) \cup \left(\bigcup_{i=1}^t F_i\right)$$
is also tight-cycle-free, where $G_i$ is the accompanying hypergraph of $F_i$, and $e(H'') \ge e(H) + t$.

Note that $F$ is contained in a complete $k$-uniform hypergraph on $m=O(k^3)$ vertices. By the result of R\"odl \cite{rodl_nibble} on the Erd\H os-Hanani Conjecture \cite{erdos-hanani}, when $n$ is sufficiently large, one can pack at least $$t\geq \left(1-o(1)\right) \frac{\binom{n-1}{k-1}}{\binom{m}{k-1}}$$
copies of $F$ in $[n-1]$ so that their $(k-1)$-shadows are disjoint. Therefore in $H''$, the number of edges is equal to
\begin{align*}
e(H'')\geq e(H) + t\geq \binom{n-1}{k-1}+ \left(1-o(1)\right) \frac{\binom{n-1}{k-1}}{\binom{m}{k-1}}\geq (1+c) \binom{n-1}{k-1},
\end{align*}
for sufficiently large $n$. Here $c$ is a positive constant which only depends on $k$ (but not $n$).
\end{proof}

\section{Concluding remarks}
In this paper, we construct $k$-uniform $n$-uniform hypergraphs that are tight-cycle-free and contains more edges than the full-$k$-star, disproving a conjecture of S\'os and Verstra\"ete. 
Below are some observations and related problems.

\begin{list}{\labelitemi}{\leftmargin=1em}

\item
To better understand the asymptotic behavior of $f_k(n)$, it naturally leads to the question of determining the Tur\'an number $ex(n, TC_\ell^k)$,
i.e., the maximum number of edges in a $TC_\ell^k$-free $k$-uniform $n$-vertex hypergraph.
If $\ell$ is not divisible by $k$, $TC_\ell^k$ is not $k$-partite. Therefore the complete $k$-partite $k$-uniform hypergraph does not contain a copy of $TC_\ell^k$ and thus $ex(n, TC_\ell^k) = \Theta(n^k)$. If $\ell$ is divisible by $k$, it is likely that $ex(n, TC_\ell^k)=\Theta(n^{c_{k,\ell}})$ for some constant $c_{k,\ell} \in (k-1, k)$, and $c_{k, \ell} \rightarrow k-1$ when $\ell \rightarrow \infty$.
A problem of Conlon (see \cite{MPS}) asks for 3-uniform hypergraphs if there is a constant $c>0$ such that $ex(n,TC_\ell^3)=O(n^{2+c/\ell})$ for all $\ell$ which are divisible by $3$. If this answer is affirmative, then it implies that $f_3(n)=O(n^{2+o(1)})$. One can easily show $ex(n, TC_{ks}^k)=\Omega(n^{k-2} \cdot ex(n, C_{2s}))$. So the $C_4$-free projective plane construction immediately gives $ex(n, TC_6^3)=\Omega(n^{5/2})$, while the best upper bound is $ex(n, TC_6^3)=O(n^{11/4})$ by Erd\H os \cite{erdos_partite}.

\item Let $H$ be a 3-uniform hypergraph with vertex set $[n]$. One can define a 3-unform hypergraph $H^*$ with the vertex set being the $2$-shadow $\partial_2(H)$ of $H$.
And every edge $e=\{i, j, k\}\in H$ defines an edge $e^*=\{ij,ik,jk\}$ in $H^*$. It is easy to see that $H^*$ is a linear 3-graph with $e(H^*)=e(H)$ and $v(H^*)=O(n^2)$. We say a loose cycle $(e^*_1,e^*_2,\cdots, e^*_k)$ in $H^*$ is {\it feasible}, if $|A_i\cap A_{i+1}|=1$ and for other $i, j$, $|A_i \cap A_j|=0$,
where $A_i:=e^*_i\cap e^*_{i+1}$ is a vertex in $H^*$ and also a $2$-subset in $[n]$. It is not hard to show that $H$ is $TC^3_{\ell}$-free if and only if $H^*$ does not contain a feasible loose cycle of length $\ell$.
Let $G$ be an $n$-vertex linear $3$-uniform hypergraph.
It is known \cite{CGJ} that if $G$ has no loose cycles of length $2k$, then $e(G)=O(n^{1+1/k})$;
and it is also known that if $G$ has no loose cycles of any length, then $e(G)=O(n)$.
If these proofs could be adapted for feasible loose cycles as well, then it immediately gives $ex(n, TC^3_{6})=O((n^2)^{1+1/3})=O(n^{8/3})$ confirming a conjecture in \cite{verstraete_survey}, and also $f_3(n)=\Theta(n^2)$.

\item The following conjecture on tight paths was asked by Kalai, as an attempt to generalize the Erd\H os-S\'os Conjecture \cite{erdos-sos}:

\begin{conjecture}
	For any integer $k$ and $\ell$, suppose $H$ is a $k$-uniform $n$-vertex hypergraph not containing a tight $k$-uniform path of length $\ell$, then $$e(H) \le \frac{\ell-1}{k} \binom{n}{k-1}.$$
\end{conjecture}

Recently, F\"uredi et al. \cite{fjkmv} prove an upper bound upper bound $e(H) \le \frac{(\ell-1)(k-1)}{k} \binom{n}{k-1}$ for this problem, improving the trivial upper bound $e(H) \le (\ell-1) \binom{n}{k-1}$. An earlier result of Patk\'os \cite{patkos} gives an upper bound $e(H) \le \sum_{j=2}^{\ell} \frac{j-1}{k-j+2}\binom{n}{k-1}=O_{\ell}(\frac{1}{k}\binom{n}{k-1})$ for $k \ge 2\ell$, which is better than the previous bound if $k$ is somewhat larger than $\ell$. It would be very interesting to further improve bounds for Kalai's Conjecture. As a side note, in \cite{GKL}, Gy\H{o}ri et al. proved that any $k$-uniform $n$-vertex hypergraph with more than $(\alpha n-k)\binom{n}{k-1}$ edges contains a tight path on $\alpha n$ vertices.

\end{list}
\medskip
\noindent \textbf{Acknowledgement. }The authors would like to thank Tao Jiang and Jacques Verstra\"ete for reading the manuscript and providing valuable comments, and Bal\'azs Patk\'os for bringing the reference \cite{patkos} to our attention.

\end{document}